\documentclass[reqno]{amsart}
\usepackage{amssymb}
\usepackage{ifpdf}
\ifpdf
\usepackage[hyperindex]{hyperref}%
\else
 \expandafter\ifx\csname dvipdfm\endcsname\relax
 \usepackage[hypertex,hyperindex,pagebackref]{hyperref}%
 \else
 \usepackage[dvipdfm,hyperindex,pagebackref]{hyperref}%
 \fi
\fi
\newtheorem{thm}{Theorem}[section]
\newtheorem{cor}{Corollary}[section]
\numberwithin{equation}{section}
\newtheorem{lem}{Lemma}[section]
\usepackage{comment}

\def\l{\lambda }

\def\a{\alpha }
\def\diag{\mathop{\rm diag}}
\def\Span{\mathop{\rm Span}}

\begin{document}

\title[Some matrix inequalities of log-majorization type]
{Some matrix inequalities of log-majorization type}

\author[B.-Y. Xi]{Bo-Yan Xi}
\address[B.-Y. Xi]{College of Mathematics, Inner Mongolia University for Nationalities, Tongliao City, Inner Mongolia Autonomous Region, 028043, China}
\email{\href{mailto: B.-Y. Xi <baoyintu78@qq.com>}{baoyintu78@qq.com}, \href{mailto: B.-Y. Xi <baoyintu78@imun.edu.cn>}{baoyintu78@imun.edu.cn}}

\author[F. Zhang]{Fuzhen Zhang}
\address[F.  Zhang]{Department of Mathematics, Nova Southeastern University, 3301 College Ave., Fort Lauderdale, FL 33314, USA}
\email{\href{mailto: F. Zhang <zhang@nova.edu>}{zhang@nova.edu}}

\begin{abstract}
The purpose of this paper is two-fold: we present some matrix
 inequalities  of log-majorization type for eigenvalues indexed
 by a sequence; we then apply our main theorem to
 generalize
and improve the Hua-Marcus' inequalities. Our results are stronger and more general
than the existing ones.
\end{abstract}



\thanks{Xi's work  was partially supported by NNSF of China grant
No.~11361038.}

\maketitle

\section{Introduction}

Let $\mathbb{C}^{m\times n}$ be the space of $m\times n$ complex matrices.
For $A\in \Bbb C^{n\times n}$, we denote the eigenvalues
of $A$ by $\lambda_1(A), \lambda_2(A),  \dots, \lambda_n(A)$
and the singular  values of $A$ by
$\sigma_1(A), \sigma_2(A),  \dots, \sigma_n(A)$. If $\lambda_1(A), \lambda_2(A),  \dots, \lambda_n(A)$
are all real, we arrange them in decreasing order:
 $\lambda_1(A)\ge\lambda_2(A)\ge \cdots\ge\lambda_n(A)$. The singular values are always
  ordered decreasingly: $\sigma_1(A)\ge\sigma_2(A)\ge\dots\ge\sigma_n(A)$.

Horn conjecture  ...

There is a large family of matrix inequalities concerning the eigenvalues and singular values of the product (including Schur or Hadamard product)
 and sum of matrices.  These inequalities may be sorted in four types:
$\sum\zeta_{_\mathcal{I}}(A+B), \, \prod\zeta_{_\mathcal{I}} (A+B), \, \sum \zeta_{_\mathcal{I}} (AB), \, \prod\zeta_{_\mathcal{I}} (AB)$, where
$A$ and $B$ are generic matrices, and $\zeta_{_\mathcal{I}}(\cdot)$ represents selected
 eigenvalues or singular values indexed by a sequence ${\mathcal{I}}$. The inequalities in sum $\sum$ are usually called majorization type (\cite[p.\,45]{HLP},
\cite{M-O-A-2011}), while the ones in product {$\prod$} are referred to as logarithmic (log-)
majorization type.
See \cite[p.\,16]{M-O-A-2011}, \cite{AndoLogMaj94, HiLim17}).
For example, the following inequalities \cite{Fiedler-1971}
 (or  \cite[G.2.a, p.\,333]{M-O-A-2011})
 \label{lem-Fiedler}  are log-majorization type with $\mathcal{I}=\{k, \dots, n\}$: for each $k=1, \dots, n$,
\begin{equation}\label{lem-Fiedler}
 \prod\limits_{t=k}^n\lambda_{t}(A+B)\ge\prod\limits_{t=k}^n[\lambda_{t}(A)+\lambda_{t}(B)]\ge
 \prod\limits_{t=k}^n\lambda_{t}(A)+\prod\limits_{t=k}^n \lambda_{t}(B),
\end{equation}
where   $A$ and $B$ are $n\times n$ positive semidefinite   matrices.

Closely related inequalities of the same type, due to  Oppenheim   \cite{Oppenheim-AMM-1954} (or \cite[F.2, p.\,685]{M-O-A-2011}), are, for $n\times n$ positive semidefinite $A$ and $B$, and  for each $k=1, \dots, n,$
\begin{align}\label{Mink-2}
 \left [ \prod\limits_{t=k}^n\lambda_{t}(A + B)\right]^{\frac{1}{n-k+1}}
   \ge& \left[ \prod\limits_{t=k}^n\lambda_{t}(A)\right]^{\frac{1}{n-k+1}} + \left [ \prod\limits_{t=k}^n\lambda_{t}(B)\right]^{\frac{1}{n-k+1}},
\end{align}
which give the Minkowski inequality by setting $k=1$
(\cite[p.\,685]{M-O-A-2011} or \cite[p.\,215]{ZFZbook11})
\begin{align}\label{Mink-1}
\left[\det(A+B)\right]^{\frac{1}{n}}\ge\left[\det(A)\right]^{\frac{1}{n}}+
 \left[\det(B)\right]^{\frac{1}{n}}.
\end{align}

In \cite{XiZ19}, we showed   majorization  inequalities of $\sum \zeta_{_\mathcal{I}} (AB)$  for Hermitian matrices $A$ and $B$ with an arbitrary index set  $\mathcal{I}$. In this paper, we present
some log-majorization  inequalities of $\prod \zeta_{_\mathcal{I}} (A+B)$ with any index set
$\mathcal{I}$  for positive semidefinite matrices $A$ and $B$, extending (\ref{lem-Fiedler}) and (\ref{Mink-2}).
Our   theorems  are more general and stronger than some existing results. We  apply
main results to improve the Hua-Marcus inequalities  for   contractive matrices.

\medskip

\section{Some lemmas}

Let $X^*$  denote the conjugate transpose of matrix or vector $X$.
For a square matrix $A$, we write $A\ge 0$ if $A$ is
positive semidefinite and $A>0$ if $A$ is positive definite. For Hermitian matrices $A, B\in \Bbb C^{n\times n}$,
we write $A\ge B$ if $A-B\ge 0$.



Through the paper,
let $n$ be a positive integer. Let $\mathcal{I}=\{i_1, \dots, i_k\}$, where
$1\le i_1<\cdots<i_k\le n$ is  any subsequence of $1, \dots, n$,
$k=1,  \dots, n$.  If $A\in \mathbb{C}^{n\times n}$ is  positive semidefinite, then
$\l_{i_1}(A)\geq   \cdots \geq \l_{i_k}(A)$  are $k$ eigenvalues of $A$
indexed by $\mathcal{I}=\{i_1, \dots, i_k\}$.  For $t=1,  \dots, k$,
 setting $i_t=t$ gives
the first $k$ largest eigenvalues of $A$; putting $i_t=n-k+t$ gives
the last $k$ smallest eigenvalues of $A$.

\begin{lem}[Hoffman {\cite[Cor.~2.5]{Amir-Moez-DMJ-1956}}] \label{lem1-Amir}
Let $A\in \mathbb{C}^{n\times n}$ be  positive semidefinite. Then
\begin{eqnarray*}
\prod\limits_{t=1}^k\lambda_{i_t}(A)
 & = & \max\limits_{\mathbb{S}_1\subset\cdots\subset \mathbb{S}_k\subset \mathbb{C}^n\hfill \atop  \dim \mathbb{S}_t=i_t}
  \min\limits_{ {x}_t\in \mathbb{S}_t,  ({x}_r,{x}_s)=\delta_{rs}\hfill\atop U_k=\left(x_{1},\dots,x_{k}\right)}
  \det({U_k^*AU_k}),
\end{eqnarray*}
where $\delta_{rs}$ is the Kronecker delta, i.e., $\delta_{rs}=1$ if $r=s$, or 0 otherwise.
\end{lem}

\begin{lem}[Lidski\v{i} \cite{Lid50}]\label{lem2-LWZ}
Let $A, B \in \mathbb{C}^{n\times n}$ be positive semidefinite. Then
\begin{equation}\label{lem2-LWZ-2}
\prod\limits_{t=1}^k\lambda_{i_t}(A)\lambda_{n-t+1}(B)
\le \prod\limits_{t=1}^k\lambda_{i_t}(AB)\le\prod\limits_{t=1}^k\lambda_{i_t}(A)\lambda_t(B).
\end{equation}
\end{lem}
More general inequalities of (\ref{lem2-LWZ-2}) for singular values are
due to Gel’fand and Naimark (see, e.g., \cite[p.\,340]{M-O-A-2011}). The inequalities
on the left-hand side of (\ref{lem2-LWZ-2})
 can be found explicitly in \cite{Wang-Zhang-LAA-1992}.

\begin{lem}\label{lem5Z}
Let $a_1, \dots, a_n$ and $b_1, \dots, b_n$ be nonnegative real  numbers.
Then
\begin{equation}\label{lem2.3-1}
\prod\limits_{t=1}^n
(a_t+b_t)\geq \prod\limits_{t=1}^n
a_t+\prod\limits_{t=1}^n
b_t +(2^n-2) \left [ \prod\limits_{t=1}^n
(a_tb_t)\right ]^{\frac12}.
\end{equation}
In particular, for
$x, y\geq 0$,
\begin{align}\label{lem-5-1}
&\big(x^{\frac{1}{n}}+y^{\frac{1}{n}}\big)^n
\ge x+y+\left(2^n-2\right)(xy)^{\frac{1}{2}}.
\end{align}
\end{lem}
\begin{proof}
Use induction on $n$ and apply the arithmetic-geometric mean inequality.
\end{proof}

We remark  that  Lemma \ref{lem5Z} implies immediately
 a result of Hartfiel  \cite{Hartfiel-PAMS-1973}:
\begin{equation}\label{Hartfiel-1}
\det(A + B)\ge \det A + \det B+(2^n-2)[\det A\det B]^{\frac12}
\end{equation}
because of the fact that any two positive semidefinite matrices
of the same size are simultaneously ${}^*$-congruent to diagonal matrices  (see, e.g.,  \cite[p.\,209]{ZFZbook11}).


\section{Main Results}

We begin with a result that is important to the proof of our main theorem. 

\begin{lem}\label{lem2-N}
Let
$D={\rm{diag}}\big(\lambda_{1},\lambda_{2},\dots,\lambda_{n}\big)$, where
  $\lambda_1\ge\lambda_2\ge \cdots\ge\lambda_n\ge0$, and let
$U_k=\left(u_{1},\dots,u_{k}\right)=\left(u_{ij}\right)$ be a partial isometry,
i.e., $U_k$ is an  $n\times k$
matrix such that $U^*_kU_k=I_k$. If $m\geq 1$ and  $u_{i1}=0$ for all $i> m$, then
\begin{align}\label{lem2-N-1}
& \det(U^*_kDU_k)\ge\lambda_{m}\det\left [\left(u_{2},\dots,u_{k}\right)^*D_m\left(u_{2},\dots,u_{k}\right)\right ],
\end{align}
where $D_{m}={\rm{diag}}\big(\underbrace{\lambda_{m},\dots,\lambda_{m}}_{m},
\lambda_{m+1},\dots,\lambda_{n}\big).$
\end{lem}

\begin{proof}
Let $V=\left(u_{2},\dots,u_{k}\right)$. Since $D\ge D_{m}$
and $u_{i1}=0$ for $i>m$, we have
\begin{eqnarray*}
 \det(U_k^*DU_k)& \ge & \det(U_k^*D_{m}U_k)
=\det\begin{pmatrix}
 \lambda_{m}u^*_1u_1  & \lambda_{m}u^*_1V \\
 \lambda_{m}V^*u_1 & V^*D_{m}V
\end{pmatrix}\\
& =&\det\begin{pmatrix}
 \lambda_{m}  & 0 \\
 0 & V^*D_{m}V
\end{pmatrix}
=\lambda_{m}\det\left(V^*D_mV\right)\\
& = & \lambda_{m}\det\left [\left(u_{2},\dots,u_{k}\right)^*D_m\left(u_{2},\dots,u_{k}\right)\right ].
\end{eqnarray*}
\end{proof}

We are ready to present our main result.

\begin{thm}\label{thm10}
Let $A$ and $B$ be $n\times n$ positive semidefinite matrices. Then
\begin{eqnarray}\label{thm10-1z}
\prod\limits_{t=1}^k\lambda_{i_t}^{\frac1k}(A + B)
\ge\prod\limits_{t=1}^k\lambda_{i_t}^{\frac1k}(A) + \prod\limits_{t=1}^k\lambda_{n-t+1}^{\frac1k}(B)
\end{eqnarray}
and
\begin{eqnarray}\label{thm10-1}
{\small \prod\limits_{t=1}^k\lambda_{i_t}(A + B)
\ge\prod\limits_{t=1}^k\lambda_{i_t}(A) + \prod\limits_{t=1}^k\lambda_{n-t+1}(B)
     +(2^k-2)\prod\limits_{t=1}^k\left[\lambda_{i_t}(A)\lambda_{n-t+1}(B)\right]^{\frac{1}{2}}\!.\;}
\end{eqnarray}
\end{thm}

\begin{proof}
If $n=1$, the inequalities become equalities and
  hold trivially. Let  $n\ge 2$. By   spectral decomposition,
  there exists a unitary matrix
$U=(u_1,  u_2, \dots, u_n)\in \mathbb{C}^{n\times n} $, where
$u_1,  u_2, \dots, u_n$ are orthonormal eigenvectors associated with
$\lambda_1(A), \lambda_2(A), \dots,$ $\lambda_n(A)$, respectively,   such that
\begin{equation*}
A=U\,{\rm{diag}}(\lambda_1(A), \lambda_2(A), \dots,\lambda_n(A))U^*.
\end{equation*}

For each $i_t$ in the sequence $i_1<i_2<\cdots < i_k$,
let
$\mathbb{S}^{(0)}_{t}={\rm{Span}}(u_{1},u_{2}, \ldots,u_{i_t}),$ $ t=1, 2, \dots,  k.$
Then
$\mathbb{S}^{(0)}_1\subset \mathbb{S}^{(0)}_2\subset \cdots \subset \mathbb{S}^{(0)}_k$
 and $\dim \big(\mathbb{S}^{(0)}_t\big)={i_t},$ $t=1, 2, \dots, k.$

Let $\{x_1, \dots, x_k\}$ be any set of orthonormal vectors,
where  $x_t\in\mathbb{S}^{(0)}_t,$ $t=1, 2,  \dots, k.$ Let
$x_t=a_{1t}u_1+\cdots +a_{i_t t} u_{i_t}+0u_{i_t+1}+\dots + 0 u_{i_k}=(u_1, \dots, u_{i_k})\a_t$,
where $\a_t = (a_{1t}, \dots, a_{i_{t} t}, 0, \dots, 0)^{\tiny T}$
(here ${}^{\tiny T}$ is for transpose),
$t=1, 2, \dots, k$. Then
$$U_k=(x_1, \dots, x_k)= (u_1, \dots, u_{i_k})(\a_1, \dots, \a_k).$$

Let $V_k=(\a_1, \dots, \a_k)\in \Bbb C^{i_k\times k}$.  Since
$\{x_1, \dots, x_k\}$ and $\{u_1, \dots, u_{i_k}\}$ are orthonormal sets,
we see that
 $V_k^*V_k=I_k$.  So $V_k$ is a partial isometry, and
 the components of the first column of $V_k$
  are zeros except the first $i_1$ components (i.e.,
 the last $i_k-i_1$  components of $\a_1$ are all equal to zero).

 In Lemma \ref{lem2-N}, setting $n=i_k$, $m=i_1$,   $U_k=V_k$,  
and applying (\ref{lem2-N-1}) to  the $i_k\times i_k$ matrix
 $D_{i_1}=\diag (\underbrace{\l_{i_1}(A),\dots,  \l_{i_1}(A)}_{i_1}\,, \l_{{i_1}+1}(A), \dots, \l_{i_k}(A)),$
 we obtain
\begin{eqnarray*}
\det (U^*_kAU_k) & = & \det \big  [ (\a_1, \dots, \a_k)^*(u_1, \dots, u_{i_k})^*A (u_1, \dots, u_{i_k})(\a_1, \dots, \a_k)\big   ]\\
  & = &  \det \big [(\a_1, \dots, \a_k)^*\diag (\l_1(A), \dots, \l_{i_k}(A))(\a_1, \dots, \a_k)\big]\\
    &=  &  \det \big [ V_k^*  \diag (\l_1(A), \dots, \l_{i_k}(A)) V_k\big ]\\
     & \geq  &   \l_{i_1}(A) \det  \big [(\a_2, \dots, \a_k)^*D_{i_1}(\a_2, \dots, \a_k) \big ].
  \end{eqnarray*}

Repeatedly using Lemma \ref{lem2-N}, we get
\begin{equation}\label{KeyLemma}
 \det(U_k^*AU_k)\ge \prod\limits_{t=1}^k\lambda_{i_t}(A).
\end{equation}

A minimax eigenvalue result of Fan (see  \cite{Fan1949, Fan1950}
or \cite[A.3.a, p.\,787]{M-O-A-2011})  ensures
\begin{align*}
  \det(U_k^*BU_k)\ge \min_{U^*U=I_k}\det (U^*BU)=\prod\limits_{t=1}^k\lambda_{n-t+1}(B).
\end{align*}

An application of  the Minkowski  inequality \eqref{Mink-1} together with (\ref{KeyLemma})  reveals
\begin{eqnarray*}
 \det\big [ {U_k^*(A + B)U_k}\big ] & = & \det({U_k^*AU_k + U_k^*BU_k})\\
 & \ge & \left[\big (\det({U_k^*AU_k }) \big )^{\frac{1}{k}}+\big(\det({U_k^*BU_k})\big)^{\frac{1}{k}}\right]^{k}\notag\\
 & \ge & \left [\prod\limits_{t=1}^k\lambda^{\frac{1}{k}}_{i_t}(A)+\prod\limits_{t=1}^k\lambda^{\frac{1}{k}}_{n-t+1}(B)\right]^{k}.
\end{eqnarray*}

Thus, by  Lemma \ref{lem1-Amir},   we derive  (\ref{thm10-1z}) and (\ref{thm10-1}) as follows:
\begin{eqnarray}\label{thm3-2}
\prod\limits_{t=1}^k\lambda_{i_t}(A + B)
 & = &  \max\limits_{\mathbb{S}_1\subset\cdots\subset \mathbb{S}_k\subset \mathbb{C}^n\hfill \atop  \dim \mathbb{S}_t=i_t}
  \min\limits_{ {y}_t\in \mathbb{S}_t,  \,({y}_r,{y}_s)=\delta_{rs}\hfill\atop W_k=\left(y_{1},\dots,y_{k}\right)}
  \det \big [{W_k^*(A + B)W_k}\big ]\notag\\
& \ge &   \min\limits_{ {x}_t\in \mathbb{S}^{(0)}_t\,  ({x}_r,{x}_s)=\delta_{rs}\hfill\atop U_k=\left(x_{1},\dots,x_{k}\right)}
  \det\big [{U_k^*(A + B)U_k}\big ] \notag\\
& \ge  & \left [\prod\limits_{t=1}^k\lambda^{\frac{1}{k}}_{i_t}(A)
  +\prod\limits_{t=1}^k\lambda^{\frac{1}{k}}_{n-t+1}(B)\right]^{k}\notag \\ 
  & \geq & \prod\limits_{t=1}^k\lambda_{i_t}(A) + \prod\limits_{t=1}^k\lambda_{n-t+1}(B)\notag\\
  & &
     +(2^k-2)\prod\limits_{t=1}^k\left[\lambda_{i_t}(A) \lambda_{n-t+1}(B)\right]^{\frac{1}{2}}. \notag
\end{eqnarray}
The last inequality is by (\ref{lem-5-1})
with $x= \prod\limits_{t=1}^k\lambda_{i_t}(A)$ and  $y= \prod\limits_{t=1}^k\lambda_{n-t+1}(B)$.
\end{proof}


In (\ref{thm10-1z}), letting $i_t=n-k+t$, $t=1, 2, \ldots, k$,  we arrive at   the inequalities
(\ref{Mink-2}) of Oppenhiem for the product of  $k$ smallest eigenvalues. Setting $i_t=t$, $t=1, 2, \ldots, k$, we obtain   analogous
 inequalities of \ref{Mink-2} for the product of $k$ largest eigenvalues.

\begin{cor}\label{cor2020Dec6a}
Let $A$ and $B$ be $n\times n$ positive semidefinite matrices. Then
\begin{align}\label{cor2020Dec6a}
&\prod\limits_{t=1}^k\lambda_{t}^{\frac1k}(A + B)
\ge\prod\limits_{t=1}^k\lambda_{t}^{\frac1k}(A) + \prod\limits_{t=1}^k\lambda_{n-t+1}^{\frac1k}(B).
\end{align}
\end{cor}

What follows  is    a lower bound for the
product  of any two eigenvalues of the sum in terms of the eigenvalues of individual matrices.
\begin{cor}\label{cor2020Dec6a}
Let $A$ and $B$ be $n\times n$ positive semidefinite matrices. Then
\begin{eqnarray*}
\l_i(A+B)\l_j(A+B)
& \ge & \lambda_{i}(A)\l_j(A) +\lambda_{n-1}(B)\l_n(B)\\
& &
     +2\big [\lambda_{i}(A)\l_j(A) \lambda_{n-1}(B)\l_n(B)\big ]^{\frac{1}{2}}.
     \end{eqnarray*}
\end{cor}

Remark: In view of  Fiedler's (\ref{lem-Fiedler}) and inequalities
(\ref{thm10-1}), it is tempting to have
\begin{eqnarray}\label{thm10-1z2020Dec6}
\prod\limits_{t=1}^k\lambda_{i_t}(A + B)
\ge\prod\limits_{t=1}^k \big [ \lambda_{i_t}(A) +  \lambda_{n-t+1}(B)\big ].
\end{eqnarray} 
However, this need not be true. Take
$A=B=\left ( \begin{smallmatrix}
             1 & 0 \\
             0 & 0
           \end{smallmatrix}\right )$, $k=2, i_1=1, i_2=2$. Then
           the left-hand side of (\ref{thm10-1z2020Dec6}) is 0, while the
           right-hand side is 1.  (Note: it is always true that
           $\lambda_{i}(A + B)\geq \lambda_{i}(A) +  \lambda_{n}(B)$. See, e.g.,
           \cite[p.\,274]{ZFZbook11}).

The following inequality in (\ref{cor1-21a}) is Fiedler's 1st inequality
in  (\ref{lem-Fiedler}). Inequality in (\ref{cor1-21b}) is stronger than
Fiedler's 2nd inequality. 
(\ref{cor1-21b}) is proved by  Lemma \ref{lem5Z}.
The inequalities in (\ref{cor1-21c}) are immediate from the inequalities
(\ref{thm10-1}) in the theorem.

\begin{cor}\label{cor1-2}
Let $A, B\in \mathbb{C}^{n\times n}$ be positive semidefinite matrices. Then
\begin{eqnarray}\label{cor1-21}
\prod\limits_{t=n-k+1}^n\lambda_{t}(A + B) &\ge  & \prod\limits_{t=n-k+1}^n \big  [\lambda_{t}(A)+\lambda_{t}(B)\big ] \label{cor1-21a} \\
& \ge&\prod\limits_{t=n-k+1}^n\lambda_{t}(A) + \prod\limits_{t=n-k+1}^n\lambda_{t}(B)  \label{cor1-21b} \\
&&
     +\left(2^{k}-2\right)\prod\limits_{t=n-k+1}^n\left [\lambda_{t}(A)\lambda_{t}(B)\right ]^{\frac{1}{2}}
    \notag
\end{eqnarray}
and
\begin{eqnarray}
\prod\limits_{t=1}^k\lambda_{t}(A + B)
& \ge&\prod\limits_{t=1}^k\lambda_{t}(A) + \prod\limits_{t=1}^k\lambda_{n-t+1}(B) \label{cor1-21c}\\
&&
     +\left(2^{k}-2\right)\prod\limits_{t=1}^k \left [\lambda_{t}(A)\lambda_{n-t+1}(B)\right ] ^{\frac{1}{2}}.
     \notag
\end{eqnarray}
\end{cor}

Remark: In view of (\ref{cor1-21a}), it is appealing in the display (\ref{cor1-21c} to have
$$\prod\limits_{t=1}^k\lambda_{t}(A + B)\;\; \ge  \;\; \prod\limits_{t=1}^k \big  [\lambda_{t}(A)+\lambda_{t}(B)\big ]\;\; \mbox{or}\;\;
   \prod\limits_{t=1}^k \big  [\lambda_{t}(A)+\lambda_{n-t+1}(B)\big ] . $$
But neither one is true. Let
$A=\left ( \begin{smallmatrix}
             1 & 0 \\
             0 & 0
           \end{smallmatrix}\right )$,  $B=\left ( \begin{smallmatrix}
          0& 0 \\
             0 & 1
           \end{smallmatrix}\right )$, $k=1$. Then $\l_1(A+B)=1<2= \l_1(A)+\l_1(B).$
          See the counterexample below (\ref{thm10-1z2020Dec6}) for the second  case.


\section{Hua-Marcus inequalities for contractions}

A matrix $A\in \mathbb{C}^{m\times n}$ is said to be (strictly) \emph{contractive} if $I_n-A^*A\geq 0$ ($>0$), equivalently,   the largest singular value (i.e., the spectral norm)
of $A$ is less than or equal to (resp. $<$)  $1$. We use
$\mathcal{C}_{m\times n}$ to denote the set of $m\times n$ contractive matrices and
 $\mathcal{SC}_{m\times n}$ to denote the
set of $m\times n$ strictly
contractive matrices.

In this section we apply our main theorem to derive some inequalities of Hua-Marcus  type
for contractive matrices.
We begin by citing Hua's  results in \cite{Hua-AMS-1955}. 

\begin{thm}[Hua {\cite[Theorems~1~and ~2]{Hua-AMS-1955}}] \label{thm-Hua-1}
Let $A,B\in \mathcal{SC}_{n\times n}$.  Then 
\begin{align}\label{thm-Hua-2}
\det(I-A^*A) \det(I-B^*B) + \left|\det(A-B)\right|^2\le \left|\det(I-A^*B)\right|^2.
\end{align}
Consequently,
\begin{align}\label{thm-Hua-2z}
\det(I-A^*A) \det(I-B^*B)\le \left|\det(I-A^*B)\right|^2.
\end{align}
Equality in \eqref{thm-Hua-2}  holds if and only if $A=B$.
\end{thm}

The following {\cite[Theorem ~7.18]{ZFZbook11}} is a reversal of  the Hua inequality \ref{thm-Hua-2}
  for general $n\times n$
matrices $A$ and $B$ that need not be contractive:
\begin{align}\label{thm-Zhang-2}
\left|\det(I-A^*B)\right|^2\le\det(I+A^*A)\det(I+B^*B)-\left|\det(A+B)\right|^2.
\end{align}

Marcus  \cite{Marcus-1958}   extended the Hua
inequality \eqref{thm-Hua-2z} to    the  inequalities of eigenvalues.

\begin{thm}[Marcus {\cite[Theorem]{Marcus-1958}}] \label{thm-MM}
 Let $A,B\in \mathcal{C}_{n\times n}$. Then
\begin{eqnarray}\label{thm1-MarcusZ}
\prod\limits_{t=1}^k|\lambda_{n-t+1}(I-A^*B)|^2\ge \prod\limits_{t=1}^k\big[ 1-\lambda_t(A^*A)\big]\big[1-\lambda_t(B^*B)\big].
\end{eqnarray}
\end{thm}

Since $|\det (I-A^*B)|=
\prod\limits_{t=1}^n|\lambda_{n-t+1}(I-A^*B)|=\prod\limits_{t=1}^n\sigma_{n-t+1}(I-A^*B)$,
by Weyl's log-majorization inequality (\cite{Weyl49} or \cite[p.\,317]{M-O-A-2011}), the
eigenvalues in absolute values  are log-majorized by the singular values,  we get
the stronger inequality than (\ref{thm1-MarcusZ}):
\begin{eqnarray}\label{thm1-Marcus}
\prod\limits_{t=1}^k\sigma^2_{n-t+1}(I-A^*B)\ge \prod\limits_{t=1}^k\big[ 1-\lambda_t(A^*A)\big]\big[1-\lambda_t(B^*B)\big].
\end{eqnarray}

Related   inequalities of Hua-Marcus type are seen in
 \cite{Ando-2008, Oppenheim-AMM-1954,
Paige-Styan-Wang-Zhang-2008,TF71,Xu-Xu-Zhang-LAA-2009,Xu-Xu-Zhang-LMA-2011}.


The following result is stronger and more general  than
\ref{thm1-Marcus}.

\begin{thm}\label{thm4}
Let $A, B\in \mathcal{SC}_{n\times n}$. Then
\begin{eqnarray}\label{thm4-1Z}
\prod\limits_{t=1}^k\sigma^2_{i_t}\left(I-A^*B\right)
    &  \ge  &  \prod\limits_{t=1}^k \big [1-\lambda_{t}(A^* A)\big ] \big [1-\lambda_{n-i_t+1}(B^*B) \big ] \\
     & &   +\prod\limits_{t=1}^k
     \frac{\big [1-\lambda_{t}(A^*A)\big ]\,\sigma^2_{n-t+1}(A-B)}{1-\lambda_{n-t+1}(A^*A)}\notag\\
     & & + (2^k-2)\prod\limits_{t=1}^k \big [ 1-\lambda_{t}(A^*A)\big ]  \big [ \lambda_{i_t}(F)\lambda_{n-t+1}(H)\big  ]^{\frac{1}{2}}\!,\;
\end{eqnarray}
where $F=I-B^*B$ and $ H=(A-B)^*(I-AA^*)^{-1}(A-B).$

\end{thm}

\begin{proof} For $A, B\in \mathcal{SC}_{n\times n}$,
let
$$
F=I-B^*B\;\;\; \mbox{and}\;\;\; H=(A-B)^*(I-AA^*)^{-1}(A-B).$$
 Then (see \cite[Theorem 1]{Hua-AMS-1955} or \cite[pp.\,230-231]{ZFZbook11})
\begin{equation}\label{EqZ:iden1}
F+H= (I-B^*A)(I-A^*A)^{-1}(I-A^*B).
\end{equation}
Applying   the left inequality of  ~~\eqref{lem2-LWZ-2} with $i_t=n-t+1$  to $H$  reveals
\begin{eqnarray}\label{thm4-4}
 \prod\limits_{t=1}^k\lambda_{n-t+1}(H)
 & = & \prod\limits_{t=1}^k \lambda_{n-t+1}\big [ (I-AA^*)^{-1} (A-B)(A-B)^*\big ] \notag \\
  & \ge & \prod\limits_{t=1}^k\lambda_{n-t+1}\big [ (I-AA^*)^{-1} \big ] \sigma^2_{n-t+1}\big(A-B\big)\notag \\
 & = & \prod\limits_{t=1}^k[1-\lambda_{n-t+1}(A^*A)]^{-1}\sigma^2_{n-t+1}\big(A-B\big).\label{Eq:4.6Nov20}
\end{eqnarray}
Applying   the right inequality of
 ~~\eqref{lem2-LWZ-2} to $F+H$ in the product form in (\ref{EqZ:iden1}) yields
\begin{align}\label{thm4-4}
\prod\limits_{t=1}^k\lambda_{i_t}\left(F+H\right)\leq  \prod\limits_{t=1}^k\left[1-\lambda_{t}(A^*A)\right]^{-1}\sigma^2_{i_t}\left(I-A^*B\right).
\end{align}
Thus
\begin{align}\label{thm4-4}
\prod\limits_{t=1}^k\sigma^2_{i_t}\left(I-A^*B\right)
\ge \prod\limits_{t=1}^k\left[1-\lambda_{t}(A^*A)\right]\lambda_{i_t}\left(F+H\right).
\end{align}

By inequalities  (\ref{thm10-1})   and (\ref{Eq:4.6Nov20}), we obtain
\begin{eqnarray}
\lefteqn{\prod\limits_{t=1}^k\lambda_{i_t}\left(F+H\right)}\notag   \\
  & \geq &  \prod\limits_{t=1}^k\lambda_{i_t}(F) + \prod\limits_{t=1}^k\lambda_{n-t+1}(H)
     +(2^k-2)\prod\limits_{t=1}^k\left[\lambda_{i_t}(F)\lambda_{n-t+1}(H)\right]^{\frac{1}{2}} \notag \\
&   \ge &  \prod\limits_{t=1}^k\left[1-\lambda_{n-i_t+1}(B^*B)\right]
 + \prod\limits_{t=1}^k\left[1-\lambda_{n-t+1}(A^*A)\right]^{-1}\sigma^2_{n-t+1}(A-B)\notag \\
& &  + (2^k-2)\prod\limits_{t=1}^k\left[\lambda_{i_t}(F)\lambda_{n-t+1}(H)\right]^{\frac{1}{2}}.\label{Eq4.9Nov20}
\end{eqnarray}
Combining (\ref{thm4-4}) and (\ref{Eq4.9Nov20}), we obtain the desired inequalities.
\end{proof}

Setting $k=n$ in Theorem~\ref{thm4} gives a stronger version of
the Hua's  inequality. 
\begin{cor}\label{cor2.1}
Let $A, B\in \mathcal{SC}_{n\times n}$. Then
\begin{align*}
&\left|\det\left(I-A^*B\right)\right|^2 \ge\det\left(I-A^*A\right)\det\left(I-B^*B\right)
 +\left|\det\left(A-B\right)\right|^2\notag\\
   &\qquad \qquad\qquad \qquad+(2^n-2)\big[\det\left(I-A^*A\right)\det\left(I-B^*B\right)\big]^{\frac{1}{2}}\left|\det\left(A-B\right)\right|.
\end{align*}
\end{cor}

Below is a reversal inequality of the previous theorem.

\begin{thm}\label{thm5}
Let $A,B\in \mathbb{C}^{n\times n}$.  Then
\begin{eqnarray}
\prod\limits_{t=1}^k\sigma^2_{i_t}\big(I-A^*B\big) & \leq &
\prod\limits_{t=1}^k\big[1+\lambda_{i_t}(A^*A )\big]\big[1+\lambda_{t}(B^*B)\big]-\prod\limits_{t=1}^k\sigma^2_{n-t+1}\big(A+B\big) \notag \\
& & - (2^k-2)\prod\limits_{t=1}^k\sigma_{i_t}\big(I-A^*B\big)\sigma_{n-t+1}\big(A+B\big).
\label{ReversalLast}
\end{eqnarray}
\end{thm}

\begin{proof}
For  $A, B\in \mathbb{C}^{n\times n}$, the  following matrix identity holds   (see, e.g., \cite[pp.\,228]{ZFZbook11}):
 $$
I+A^*A=P+Q$$
where
$$
P=(A+B)^*(I+BB^*)^{-1}(A+B),$$
$$Q=(I-A^*B)(I+B^*B)^{-1}(I-A^*B)^*.
$$
By Lemma ~\ref{lem2-LWZ}, we have
\begin{align}
\prod\limits_{t=1}^k\lambda_{n-t+1}(P)&\ge \prod\limits_{t=1}^k\big[1+\lambda_{t}(BB^*)\big]^{-1}\sigma^2_{n-t+1}\big(A+B\big), \notag\\
 \prod\limits_{t=1}^k\lambda_{i_t}(Q)
  & \ge \prod\limits_{t=1}^k\big[1+\lambda_{t}(B^*B)\big]^{-1}\sigma^2_{i_t}\big(I-A^*B\big).\label{Eq:thm4.6.1}
\end{align}
Using (\ref{thm10-1}) in  Theorem \ref{thm10}, we derive
\begin{eqnarray*}
\lefteqn{\prod\limits_{t=1}^k\big[1+\lambda_{i_t}(A^*A )\big]=\prod\limits_{t=1}^k\lambda_{i_t}(P+Q))} \\
   & \ge &\prod\limits_{t=1}^k\lambda_{n-t+1}(P)
     +\prod\limits_{t=1}^k\lambda_{i_t}(Q)+(2^k-2)
\prod\limits_{t=1}^k[\lambda_{n-t+1}(P)\lambda_{i_t}(Q)]^{\frac{1}{2}}\\
   & \ge & \prod\limits_{t=1}^k\big[1+\lambda_{t}(B^*B)\big]^{-1}
\bigg [\prod\limits_{t=1}^k\sigma^2_{i_t}\big(I-A^*B\big)+\prod\limits_{t=1}^k\sigma^2_{n-t+1}\big(A+B\big) \\
   & &  +(2^k-2)\prod\limits_{t=1}^k\sigma_{i_t}\big(I-A^*B\big)\sigma_{n-t+1}\big(A+B\big)\bigg  ].
\end{eqnarray*}
 Multiplying by $\prod\limits_{t=1}^k\big[1+\lambda_{t}(B^*B)\big]$,
we obtain the desired inequalities (\ref{ReversalLast}).
\end{proof}

As (\ref{thm-Zhang-2}) is a reversal of Hua's (\ref{thm-Hua-2}), the following result, as a special case of the Theorem \ref{thm5} by setting $k=n$,
may be viewed as a counterpart of Corollary \ref{cor2.1}. 

\begin{cor}\label{cor3.2}
Let $A,B\in \mathbb{C}^{n\times n}$. Then
\begin{eqnarray*}
\big|\det\left(I-A^*B\right)\big|^2 & \le  & \det\big(I+A^*A\big)\det\left(I+B^*B\right)-\big|\det\left(A+B\right)\big|^2 \\
  & & -(2^n-2)\big|\det\left(I-A^*B\right)\big|\big|\det\left(A+B\right)\big|.
\end{eqnarray*}
Consequently,
\begin{align*}
\big|\det\left(I-A^*B\right)\big|^2&\le\det\left(I+A^*A\right)\det\left(I+B^*B\right)
 -\big|\det\left(A+B\right)\big|^2.
\end{align*}
\end{cor}


\begin{thebibliography}{99}

\bibitem{Amir-Moez-DMJ-1956}
A.R. Amir-Mo\'{e}z, \emph{Extreme properties of eigenvalues of a Hermitian transformation
 and singular values of the sum and product of linear transformations}.
 Duke Math. J. \textbf{23} (1956), no.~3, 463\nobreakdash--476.

\bibitem{Ando-2008}
T. Ando, \emph{Positivity of operator-matrices of Hua-type}.
Banach J. Math. Anal.  \textbf{2} (2008), no.~2, 1\nobreakdash--8.

\bibitem{AndoLogMaj94}
 T. Ando, \emph{Log majorization and complementary Golden-Thompson type inequalities.}
Linear Algebra Appl. 197/198 (1994), 113--131.

 \bibitem{Fan1949}
 K. Fan,
 \emph{On a Theorem of Weyl Concerning Eigenvalues of Linear Transformations I.}
 PNAS November 1, 1949,  35 (11) 652-655.

 \bibitem{Fan1950}
 K. Fan,
 \emph{On a Theorem of Weyl Concerning Eigenvalues of Linear Transformations II}.
  PNAS January 1, 1950,  36 (1) 31-35.

\bibitem{Fiedler-1971}
M. Fiedler, \emph{Bounds for the determinant of the sum of hermitian matrices}. Proc. Amer. Math. Soc.,
\textbf{30} (1971), no.~1, 27\nobreakdash--31.

\bibitem{HLP} G.H. Hardy, J.E. Littlewood and G. P\'{o}lya,
{\em Inequalities,}
Cambridge University Press, New York, 1934, 1st Ed.; 1952, 2nd Ed.; Reprint, 1994.

\bibitem{Hartfiel-PAMS-1973}
D.J. Hartfiel, \emph{An extension of Haynsworths determinant inequality}. Proc. Amer. Math. Soc.,
\textbf{41} (1973), 41:463\nobreakdash--465.

\bibitem{HiLim17}
F. Hiai and Y.  Lim,
\emph{Log-majorization and Lie-Trotter formula for the Cartan barycenter on probability measure spaces.}
J. Math. Anal. Appl. 453 (2017), no. 1, 195--211.

\bibitem{Hua-AMS-1955}
L.-K. Hua, \emph{Inequalities involving determinants}. Acta Math. Sinica 5 (1955), 463\nobreakdash--470 (in Chinese). See also Transl.
Amer. Math. Soc. Ser. II \textbf{32} (1963), 265\nobreakdash--272.

\bibitem{Lid50}
V.B. Lidski\v{i},
\emph{The proper values of the sum and product of symmetric matrices} (in Russian).
 Doklady Akad. Nauk SSSR vol. \textbf{75} (1950) 769\nobreakdash--772.

\bibitem{Marcus-1958}
M. Marcus,  \textit{On a determinantal inequality}. Amer. Math. Monthly \textbf{65} (1958), no. 4, 266\nobreakdash--268.


\bibitem{M-O-A-2011}
A.W. Marshall, I. Olkin, and B.C. Arnold, \textit{Inequalities: Theory of Majorization and Its Application} (Second Edition). Springer, New York, 2011.


\bibitem{Oppenheim-AMM-1954}
A. Oppenheim, \emph{Inequalities connected with definite Hermitian
forms II}.  Amer. Math. Monthly \textbf{61} (1954), 463\nobreakdash--466.

\bibitem{Paige-Styan-Wang-Zhang-2008}
C.C. Paige, G.P.H. Styan, B.-Y. Wang, and F. Zhang, \emph{Huas matrix equality
and Schur complements}. Internat. J. Inform. Syst. Sci.,\textbf{4} (2008), no.1, 124\nobreakdash--135.

\bibitem{TF71}
R.C. Thompson and L.J. Freede,
\emph{On the eigenvalues of sums of Hermitian matrices}.
Linear Algebra  Appl. \textbf{4} (1971) 369--376.

\bibitem{Wang-Zhang-LAA-1992}
B.-Y. Wang and F. Zhang, \textit{Some inequalities for
 the Eigenvalues of the Product of Positive Semidefinite Hermitian Matrices}.
 Linear Algebra  Appl. \textbf{160} (1992) 113\nobreakdash--118.


\bibitem{Weyl49}
H. Weyl,  {\em Inequalities between two kinds of eigenvalues of a
linear transformation.} Proc. Nat. Acad. Sci. U.S.A. 35 (1949), 408--411.


\bibitem{Wielandt-PAMS-1955}
H. Wielandt, \textit{An extremum property of sums of eigenvalues}.
Proc. Amer. Math. Soc. \textbf{6} (1955), no.~1, 106\nobreakdash--110.

\bibitem{XiZ19} B.Y. Xi and F. Zhang,
\textit{Inequalities for selected eigenvalues
of the product of matrices.}
Proc. Amer. Math. Soc.
\textbf{147} (2019), no.~9, 3705\nobreakdash--3713.
https://doi.org/10.1090/proc/14529.


\bibitem{Xu-Xu-Zhang-LAA-2009}
C.-Q. Xu, Z.-D. Xu, F. Zhang, \textit{Revisiting Hua-Marcus-Bellman-Ando inequalities on contractive matrices}.
 Linear Algebra  Appl. \textbf{430} (2009), no. 5-6, 1499\nobreakdash--1508.

\bibitem{Xu-Xu-Zhang-LMA-2011}
G.-H. Xu, C.-Q. Xu, and F. Zhang, \textit{Contractive matrices of Hua type}. Linear
Multilinear Algebra \textbf{59} (2011), no. 2, 159\nobreakdash--172.

\bibitem{ZFZbook11}
F. Zhang, \emph{Matrix Theory: Basic Results and
Techniques} (Second edition). Springer, New York,   2011.

\end{thebibliography}
\end{document}